\documentclass[11pt,reqno]{amsart}

\usepackage{eucal,mathrsfs}

\usepackage[utf8]{inputenc}

\numberwithin{equation}{section}

\usepackage[a4paper,margin=1.2in]{geometry}
                		
\usepackage{graphicx}				

\usepackage{eucal,mathrsfs}

\usepackage{amssymb}

\usepackage[utf8]{inputenc}

\usepackage[polish, english]{babel}

\usepackage[T1]{fontenc}

\usepackage{mathtools}

\usepackage{tikz-cd}

\usepackage{amsthm}

\usepackage[toc]{appendix}

\usepackage{calligra}

\usepackage{hyperref}

\usepackage{mathabx}

\newtheorem{thm}{Theorem}[section]
\newtheorem{lem}[thm]{Lemma}

\newtheorem{prop}[thm]{Proposition}
\newtheorem{ass}[thm]{Assumption}
\newtheorem{ques}[thm]{Question}
\theoremstyle{remark}

\newtheorem{rmk}[thm]{Remark}
\newtheorem{ex}[thm]{Example}

\newcommand{\cD}{{\mathcal D}}

\newcommand{\CC}{{\mathbb C}}
\newcommand{\FF}{{\mathbb F}}

\newcommand{\QQ}{{\mathbb Q}}

\newcommand{\ZZ}{{\mathbb Z}}

\newcommand{\mm}{{\mathfrak m}}

\newcommand{\pp}{{\mathfrak p}}

\newcommand{{\D}}{{\mathscr{D}_{X}}} %SHEAF OF DIFFERENTIAL OPERATORS%

\title[Continuity of differential operators for nonarchimedean Banach algebras]{Continuity of differential operators for nonarchimedean Banach algebras} %TYTUŁ%

\author{Feliks Rączka}

\address{Institute of Mathematics, University of Warsaw, ul.\ Banacha 2, 02-097, Warsaw, Poland}

  \address{Institute of Mathematics, Polish Academy of Sciences, ul.\ Śniadeckich 8,
    \newline\indent 00-656 Warsaw, Poland
  }

\email{fraczka@impan.pl}

\date{\today}

\begin{document}

 %Data??%

\begin{abstract}
Given a nonarchimedean field $K$ and a commutative, noetherian, Banach $K$-algebra $A$, we study continuity of $K$-linear differential operators (in the sense of Grothendieck) between finitely generated Banach $A$-modules. When $K$ is of characteristic zero we show that every such operator is continuous if and only if $A/\mm$ is a finite extension of $K$ for every maximal ideal $\mm\subset A$.
\end{abstract}

\maketitle

\section{Introduction}\label{Introduction}

A classical result in nonarchimedean functional analysis asserts that if $A$ is a noetherian Banach algebra  over some non-trivially valued nonarchimedean field $K$, then every finitely generated $A$-module has a unique structure of a Banach $A$-module and every homomorphism of finitely generated $A$-modules is continuous. Since homomorphisms may be regarded as differential operators (in the sense of Grothendieck, which we recall in Section \ref{Differential preliminaries}) of order zero, it is natural to pose the following.
\begin{ques}
    Let $A$ be a commutative, noetherian Banach $K$-algebra over some nonarchimedean, non-trivially valued field $K$. Under which algebraic conditions imposed on $A$ is it true that every $K$-linear differential operator between every two finitely generated $A$-modules is continuous?
\end{ques}

\noindent
In this paper we give a complete answer to the above question when $\textnormal{char }K=0$.

\begin{thm}\label{MainThm1}
Let $K$ be a non-trivially valued nonarchimedean field of characteristic zero and let $A$ be a commutative, noetherian, Banach $K$-algebra. The following conditions are equivalent:
\begin{enumerate}

\item For any two finitely generated $A$-modules $M,\ N$, every $K$-linear differential operator $P:M\to N$ is continuous with respect to the unique Banach $A$-module topologies on $M$ and $N$.

\item $[A/\mm:K]<\infty$ for every maximal ideal $\mm\subset A$.

\end{enumerate}
Moreover, the implication (2)$\implies$(1) remains true when $\textnormal{char }K=p>0$.
\end{thm}

\noindent
In fact, we prove more than is stated above. Theorem \ref{MainThm1} is implied by the following, more precise statement.

\begin{thm}\label{MainThm2}
Let $K$ be a non-trivially valued nonarchimedean field of characteristic zero, let $A$ be a commutative, noetherian, Banach $K$-algebra, and let $M,N$ be two finitely generated (Banach) $A$-modules. The following conditions are equivalent:
\begin{enumerate}

\item There exists a discontinuous $K$-linear differential operator $P:M\to N$.

\item There exists a maximal ideal $\mm\subset A$ such that $\mm\in\textnormal{Supp}(M)\cap\textnormal{Ass}_{A}(N)$ and $[A/\mm:K]=\infty$.
\end{enumerate}
Moreover, the implication (1)$\implies$(2) remains true when $\textnormal{char }K=p>0$.
\end{thm}
\begin{ex}[cf. Remark \ref{Tate_vs_Berkovich}]
If, in the context of Theorem \ref{MainThm1}, $A$ is an affinoid $K$-algebra in the sense of Tate, then it is well known that condition (2) always holds and therefore all differential operators are continuous.
\end{ex}
There are at least two motivations to study the above problem. The first one (which we discuss in greater detail in Section \ref{Singer-Wermer section}), is that the analogous problem (without the noetherianity assumption) has been studied very intensively in the classical case of complex Banach algebras. It is therefore interesting to see how the classical results translate to the nonarchimedean setting. In fact, in the proofs we use various tools from the branch of complex functional analysis called \textit{automatic continuity theory}. It turns out that these techniques combined with the noetherianity assumption allow one to prove very strong results about automatic continuity of differential operators. We consider this observation the main contribution of this paper to the field of nonarchimedean functional analysis.

\medskip

The second motivation, which is more pragmatic and also partially justifies the noetherianity assumption, comes from the theory of nonarchimedean $\mathscr{D}$-modules. This theory has been studied classically by many authors (see for example Mebkhout--Narvaez Macarro \cite{Meb2}, and Berthelot \cite{Berthelot} among others) and has seen a very rapid progress in the recent years (see for example the work of Ardakov--Bode--Wadsley \cite{Ardakov1}, \cite{Ardakov2}, \cite{Ardakov3}, and the recent preprint of J.\ E.\ Rodríguez Camargo \cite{Camargo}), and we hope that our results will turn out to be useful tool within the theory, as the possibility to drop the continuity assumption from the definition of a differential operator should simplify proofs of various technical results. Also, since the geometric objects in nonarchimedean geometry are often modeled locally as the adic spectra of noetherian  $K$-Banach algebras, the noetherianity assumption in Theorems \ref{MainThm1} and \ref{MainThm2} is much less restrictive than in the classical setting.
\medskip

\noindent
\textbf{Acknowledgments} This work contains an extended version of some of the results obtained in author's Ph.D thesis. I thank my advisors: P.\ Achinger and A.\ Langer, for their help and support. I also thank the referees of the thesis: W.\ Gajda, G.\ Kapustka, and J.\ Poineau, for their helpful comments.
\medskip

\noindent
This work was supported by the project KAPIBARA funded by the European Research Council (ERC) under the European Union's Horizon 2020 research and innovation programme (grant agreement No 802787).  In the period of May--June 2024 I was also supported by the program ``II konkurs na przygotowanie rozpraw doktorskich zgodnych z problematyką POB'' which is a part of the ``Inicjatywa Doskonałości-Uczelnia Badawcza'' program funded by the University of Warsaw.

\section{The Singer--Wermer conjecture and Automatic Continuity}\label{Singer-Wermer section}

To put Theorems \ref{MainThm1} and \ref{MainThm2} and the machinery used in their proofs in a broader context, we now briefly summarize what is known about the continuity of differential operators over complex Banach algebras.
\medskip

In 1955 I.\ M.\ Singer and J.\ Wermer showed in \cite{Singer-Wermer} that if $A$ is a commutative Banach algebra over the complex numbers and $\delta:A\to A$ is a \textit{continuous} $\CC$-linear derivation then the image of $\delta$ is contained in the Jacobson radical of $A$. In the same paper they conjectured that the continuity assumption is superfluous and this so-called \textit{Singer--Wermer conjecture} fueled the research for more that thirty years until  M.\ P.\ Thomas proved it in 1988 in \cite{MP_Thomas} building on the ideas of B.\ E.\ Johnson \cite{Johnson}. It was very natural from the point of view of the Singer--Wermer conjecture to look for Banach algebras with the property that every $\CC$-linear derivation is automatically continuous. For example, by the work of P.\ C.\ Curtis,\ Jr. \cite{Curtis} and B.\ E.\ Johnson \cite{Johnson} every derivation of a (commutative) semi-simple complex Banach algebra is continuous and some other criteria for continuity of derivations where obtained by Bade--Curtis \cite{Bade-Curtis}, A.\ M.\ Sinclair \cite{Sinclair}, and Jewell--Sinclair \cite{Jewell-Sinclair} among others. These results fit into a bigger class of problems which can be described as follows. Given a complex Banach algebra $A$, two Banach $A$-modules $M,$ $N$, and a $\CC$-linear map $T:M\to N$, what are the conditions one can put on $A$, $M$, $N$, and $T$ to force that $T$ is necessarily continuous. Problems of this type and the techniques used to solve them became known under the common name of the \textit{automatic continuity theory}. We refer a reader who interested in a more detailed overview of this theory to A.\ M.\ Sinclair's monograph \cite{SinclairBook}.
\medskip

In the nonarchimedean setting problems of automatic continuity have been considered only to some extent. We refer the reader to the work of M.\ van der Put \cite{vdP}, \cite{VdP_derivations} and W.\ H.\ Schikhof \cite{Schikhof} for some classical results, and to T.\ Mihara \cite{Mihara} for a more recent progress. We remark that among the literature cited above only \cite{VdP_derivations} addresses the problem of continuity of derivations over nonarchimedean fields.

\section{algebraic preliminaries}\label{AlgebraicPleriminariesSection}

In this short section, we very briefly recall the basic constructions from commutative algebra used throughout the text. Let $A$ be a commutative ring and let $M$ be an $A$-module. Then the \textit{support} of $M$ is the set of prime ideals
\begin{equation*}
\textnormal{Supp}(M)=\left\{\pp\in\textnormal{Spec }A:M_{\pp}\neq0\right\}.
\end{equation*}
It is well known that if $M$ is finitely generated, then
\begin{equation*}
\textnormal{Supp}(M)=V\left(\textnormal{Ann}_{A}(M)\right)=\left\{\pp\in\textnormal{Spec }A:\textnormal{Ann}_{A}(M)\subset\pp\right\},
\end{equation*}
where
\[
\textnormal{Ann}_{A}(M)=\left\{a\in A:aM=0\right\}.
\]
We say that a prime ideal $\pp\subset A$ is an \textit{associated prime} of an $A$-module $N$ if there exists $y\in A$ such that 
\[
\pp=\textnormal{Ann}_{A}(y)=\left\{a\in A:ay=0\right\}.
\]
We denote
\begin{equation*}
\textnormal{Ass}_{A}(N)=\{\pp\in\textnormal{Spec }A:\pp \textnormal{ is an associated prime of }N\}.
\end{equation*}
We will need the following easy lemma. Its proof is an exercise in commutative algebra and we leave it to the reader.
\begin{lem}\label{SupCupAss}
Let $A$ be a commutative ring, let $M$ be a finitely generated $A$-module and let $N$ be an $A$-module. Then:
\begin{enumerate}
\item A maximal ideal $\mm\subset A$ is contained in $\textnormal{Supp}(M)$ if and only if there exists a surjective $A$-linear map $M\to A/\mm$.
\item A prime ideal $\pp\subset A$ is contained in $\textnormal{Ass}_{A}(N)$ if and only if there exists an injective $A$-linear map $A/\pp\hookrightarrow N$.

\end{enumerate}
\end{lem}

\section{topological preliminaries}

In this section we fix the notation from the nonarchimedean functional analysis and recall some basic facts concerning finitely generated modules over noetherian Banach algebras. We mostly follow the book \cite{BGR} of Bosch--G{\"u}ntzer--Remmert.
\medskip

Recall from \cite[Definition 1, Section 1.5.1]{BGR} the notion of a nonarchimedean valuation. A \textit{nonarchimedean field} is a field $K$ together with a nonarchimedean valuation $| \ |$, such that $K$ is complete with respect to the metric induced by this valuation. In what follows we prefer to refer to this valuation as a \textit{nonarchimedean norm} on $K$. The simplest example of a nonarchimedean field is to take $K$ to be any field and define the nonarchimedean norm on $K$ to be
\[
|x|=1\qquad \textnormal{for all }x\in K^{\times}.
\]
Such $K$ is called a \textit{trivially valued nonarchimedean field}. In what follows, we want to exclude this example, so for the rest of this note we pose the following assumption.
\begin{ass}\label{Nontrivial_Assumption}
 All nonarchimedean fields in this note are assumed to be non-trivially valued unless stated otherwise.   
\end{ass}
\noindent
Note that $K$ is non-trivially valued if and only if there exists $\pi\in K$ such that $0<|\pi|<1$.
\medskip

Let $K$ be a nonarchimedean field. We refer the reader to \cite[Section 2.8, Section 3.7]{BGR} for the definitions of a Banach $K$-vector space, Banach $K$-algebra $A$, and a Banach $A$-module. 

\begin{rmk}[Notational convention]
For the rest of this note we use the following notational convention to avoid cumbersome notation. Given a nonarchimedean field $K$ we write $| \ |$ for some fixed nonarchimedean norm on $K$. Further, given a Banach $K$-vector space $V$ we also write $|\ |$ for some fixed Banach space norm on $V$ compatible with the one on $K$. Similarly, when $A$ is a Banach $K$-algebra and $M$ is a Banach $A$-module, we write $|\ |$ for both a Banach $K$-algebra norm on $A$ and a Banach module norm on $M$ compatible with the norm on $A$.
\end{rmk}
\begin{rmk}[Banach fields]\label{Kedlaya-Remark}
If $A$ is a Banach $K$-algebra which is also a field, then it is not true in general that $A$ is itself a nonarchimedean field in the sense that the norm of $A$ is equivalent to a multiplicative one, as demonstrated by K.\ Kedlaya in \cite[Example 2.15]{Kedlaya_fields}. This subtlety will cause some technical problems in what follows.
\end{rmk}

Let $V$ be a finitely generated $K$-vector space. A choice of a basis
 \[
 V=\bigoplus_{j=1}^{n}K.e_{j}
 \]
determines a Banach $K$-vector space structure on $V$. The norm is defined as
\[
\left|\sum_{j=1}^{n}a_{j}e_{j}\right|=\max|a_{j}|.
\]
Similarly to the archimedean case we have
\begin{prop}[{\cite[Theorem 1.3.6]{Kedlaya_book}}]\label{FiniteNorms}
Let $K$ be a nonarchimedean field and let $V$ be a finite dimensional $K$-vector space. Then every two nonarchimedean $K$-norms norms on $V$ are equivalent. In particular, $V$ carries the unique Banach $K$-vector space topology.
\end{prop}

\noindent
As a corollary we obtain.
\begin{lem}\label{ClosedFiniteDimension}
Let $K$ be a nonarchimedean field and let $V$ be a Banach $K$-vector space. Then every finite dimensional subspace of $V$ is closed.
\end{lem}
\begin{proof}
It suffices to show that every finitely dimensional subspace of $V$ is complete with respect to the norm restricted from $V$. This follows immediately from Proposition \ref{FiniteNorms}.
\end{proof}
\noindent
Let us now turn to the case when $A$ is a commutative Banach $K$-algebra. In this situation the best analogue of Proposition \ref{FiniteNorms} is the following

\begin{prop}\label{Fundamental_Finite_Modules}
Let $A$ be a commutative, noetherian, Banach $K$-algebra over some nonarchimedean field $K$, and let $M$ be a finitely generated $A$-module. Then:
\begin{enumerate}
    \item There exists a unique (up to norm equivalence) Banach $A$-module structure on $M$. We call the corresponding topology on $M$ the \textnormal{canonical topology}.

    \item Every homomorphism of finitely generated $A$-modules is continuous with respect to the canonical topology.

    \item Let $M$ be a finitely generated $A$-module and let $M'\subset M$ be a submodule. Then the canonical topology on $M'$ agrees with the subspace topology. In particular every submodule of $M$ is closed.

    \item Let $\pi:F\to M$ be a surjection of finitely generated $A$-modules. Then $\pi$ is a quotient map with respect to the canonical topology.
\end{enumerate}
\end{prop}
\begin{proof}
Claims (1)-(3) follow from \cite[Propositions 1,2, and 3, Page 164]{BGR}. The last claim follows from the previous ones and the Banach open mapping theorem, since a surjective open map is a quotient map.
\end{proof}
\noindent
Note that as a consequence of (3) above we have the following pleasant property of noetherian Banach $K$-algebras.
\begin{prop}[{\cite[Proposition 2, Page 164]{BGR}}]
Let $A$ be a commutative, noetherian, Banach $K$-algebra. Then every ideal of $A$ is closed.
\end{prop}
\section{Differential preliminaries}\label{Differential preliminaries}

In this section we recall the definition and basic properties of rings and modules of differential operators (see \cite[\href{https://stacks.math.columbia.edu/tag/09CH}{Tag 09CH}]{stacks-project} for the further discussion).

\medskip

Let $K$ be a field and let $A$ be a commutative $K$-algebra. The \textit{ring of Grothendieck differential operators} on $A$, which we denote $\cD_{A/K}$, is defined as follows. We set $\cD^{\leq -1}_{A/K}=\{0\}$, and for any $n\geq0$ we define inductively
\[
\cD_{A/K}^{\leq n}=\left\{P\in\textnormal{Hom}_{K}(A,A) :[P,a]\in\cD_{A/K}^{\leq n-1}\;\textnormal{for all }a\in A\right\}.
\]
Then by the definition
\begin{equation}\label{DiffRing}
\cD_{A/K}=\bigcup_{n\geq 0}\cD_{A/K}^{\leq n}.
\end{equation}
It is standard that $\cD_{A/K}$ is a filtered ring and that we have
\[
\cD_{A/K}^{\leq 0}=A,\qquad \cD_{A/K}^{\leq 1}=A\oplus\textnormal{Der}_{K}(A).
\]
More generally, when $M,N$ are $A$-modules we can consider the $A$-module $\cD_{A/K}(M,N)$ of differential operators from $M$ to $N$, which is defined in way similar to (\ref{DiffRing}). We put $\cD_{A/K}^{\leq -1}(M,N)=\{0\}$, we define
\[
\cD_{A/K}^{\leq n}(M,N)=\left\{P\in\textnormal{Hom}_{K}(M,N) :[P,a]\in\cD_{A/K}^{\leq n-1}\;\textnormal{for all }a\in A\right\},
\]
and we set
\[
\cD_{A/K}(M,N)=\bigcup_{n\geq 0}\cD_{A/K}^{\leq n}(M,N).
\]
In the above, we identify an element $a\in A$ with an $A$-linear endomorphism of $M$ (resp. $N$) given by $m\mapsto am$ and abusing the notation slightly we define $[P,a]$ to be the $K$-linear map
\[
[P,a]:m\mapsto P(am)-aP(m).
\]
The following property of differential operators is a straightforward calculation.

\begin{lem}[{\cite[\href{https://stacks.math.columbia.edu/tag/09CJ}{Tag 09CJ}]{stacks-project}}]
Let $M_{1},M_{2},M_{3}$ be $A$-modules. If $P_{1}\in\cD_{A/K}^{\leq n_{1}}(M_{1},M_{2})$ and $P_{2}\in\cD_{A/K}^{\leq n_{2}}(M_{2},M_{3})$, then the composition $P_{2}P_{1}\in\cD_{A/K}^{\leq n_{1}+n_{2}}(M_{1},M_{3})$.
\end{lem}

As a consequence of the above, we obtain the following easy lemma which turns out to be useful for constructing discontinuous differential operators between Banach modules.

\begin{lem}\label{CompositionLemma}
Let $M,N$ be two $A$-modules and let $K\subset L$ be a field extension. Assume that we are given
\begin{enumerate}
    \item An $A$-module surjection $\pi:M\to L$,
    \item an $A$-module injection $\epsilon:L\to N$,
    \item a $K$-linear derivation $\delta:L\to L$.
\end{enumerate}
Then the composition
\begin{equation*}
    \begin{tikzcd}
        M\arrow{r}{\pi}&L\arrow{r}{\delta}&L\arrow{r}{\epsilon}&N
    \end{tikzcd}
\end{equation*}
is an element of $\cD_{A/K}^{\leq 1}(M,N)$.
\end{lem}

\section{constructions from classical functional analysis}
In this section we present classical results from the automatic continuity theory in the nonarchimedean setting and we give some refinements building on the noetherianity assumption. At the end of the section we include a short historical remark where we try to trace original authors of the presented results.
\medskip

Let $K$ be a nonarchimedean field, and let $M, N$ be two $K$-Banach spaces. Given a $K$-linear map $\varphi:M\to N$ we want to decide whether or not it is continuous. The first observation is that this property may be checked on closed subspaces of finite codimension in $M$. This result is \textit{folklore} but we provide a short proof for completeness.
\begin{lem}\label{FiniteCodimensionContinuity}
Assume that there exists a closed subspace $M'\subset M$ such that
\begin{enumerate}
\item $\dim_{K}M/M'<\infty$
\item The restriction $\varphi:M'\to N$ is continuous.
\end{enumerate}
Then $\varphi$ is continuous.
\end{lem}
\begin{proof}
First of all, every closed subspace $M'\subset M$ of finite codimension is \textit{complemented}, i.e., there exists a closed subspace $M''\subset M$ such that $M=M'\oplus M''$ as topological vector spaces. Indeed, since $M'$ has finite codimension there exists a finitely dimensional $M''$ such that $M=M'\oplus M''$ as vector spaces. To show that the topological structure on both sides is the same it suffices to show that $M''$ is a closed subspace. However, every finitely dimensional subspace is closed by Lemma \ref{ClosedFiniteDimension}.
\medskip
Now, consider the $K$-linear map
\[
\varphi':M=M'\oplus M''\to M;\qquad \varphi'(m_{1},m_{2})=(\varphi(m_{1}),0)
\]
By the initial assumption on $\varphi$ and by the above considerations $\varphi'$ is continuous. Therefore to show that $\varphi$ is continuous it suffices to show that $\varphi''=\varphi-\varphi'$ is continuous. However, since $M'$ is closed we see that the projection $\pi:M\to M/M'$ is continuous. Since $\varphi''(M')=0$ we have a commutative diagram
\[
\begin{tikzcd}
M\arrow{rd}{\varphi''}\arrow{d}{\pi}\\
M/M'\arrow{r}&M
\end{tikzcd}
\]
where the horizontal arrow is necessarily continuous as its source has finite dimension. Therefore $\varphi''$ is continuous as a composition of continuous maps and we are done.
\end{proof}
A more sophisticated tool for studying continuity is the \textit{separating space} of $\varphi$, i.e., the subspace of $N$ defined as
\begin{equation}\label{SeparatinSpace}
\mathfrak{S}(\varphi)=\left\{ y\in N :\exists \{x_{n}\}\subset M\textnormal{ with}\lim_{n\to\infty} x_{n}=0\textnormal{ and} \lim_{n\to \infty}\varphi(x_{n})=y \right\}.
\end{equation}
As explained below, the separating space can be seen as a quantitative measure of discontinuity of $\varphi$.
\begin{lem}\label{SeparatingProperties}
With the notation above the following hold:
\begin{enumerate}
    \item The morphism $\varphi$ is continuous if and only if $\mathfrak{S}(\varphi)=0$.

    \item More generally, if we consider the diagram
    \[
    \begin{tikzcd}
M\arrow{r}{\varphi}&N\arrow{r}{a}&P,
    \end{tikzcd}
    \]
    where $a:N\to P$ is a continuous $K$-linear map of Banach $K$-vector spaces then the composition $a\varphi$ is continuous if and only if $a\left(\mathfrak{S}(\varphi)\right)=0$.
\end{enumerate}
\end{lem}
\begin{proof}
The lemma is very well known in the classical setting (see for example \cite[Lemma 1.2, Lemma 1.3, Chapter 1]{SinclairBook}). The proof is based on the closed graph theorem (which is valid also in the nonarchimedean setting, see \cite[Proposition 8.5]{Schneider_book}) and therefore everything carries over \textit{mutatis mutandis} to our situation.
\end{proof}
\begin{rmk}
Even more generally, one can show that $\mathfrak{S}(a\varphi)$ is the topological closure of $a\mathfrak{S}(\varphi)$, but we will not use this fact.
\end{rmk}
The above criterion for continuity is especially useful when combined with the so-called \textit{stabilization property} of Jewell--Sinclair. We write $\textnormal{cl}(V)$ for the topological closure of a subset $V$ of some topological space.
\begin{lem}[N.\ P.\ Jewell, A.\ M.\ Sinclair]\label{Jewell-Sinclair}
Let $K$ be a nonarchimedean field and let $\varphi:M\to N$ be a $K$-linear morphism between two $K$-Banach spaces. Assume that there exist continuous $K$-linear endomorphisms $\{R_{n}\}$ and $\{L_{n}\}$ of $M$ and $N$ respectively, such that
\begin{equation}\label{SeparatingAssumption}
\varphi R_{n}-L_{n}\varphi:M\to N
\end{equation}
are continuous for all $n$. Then there exists a positive integer $k$ such that for all $n\geq k$ the equality
\begin{equation}\label{SeparatingEquality}
\textnormal{cl}\left(L_{1}\dots L_{n}\mathfrak{S}(\varphi)\right)=\textnormal{cl}\left(L_{1}\dots L_{k}\mathfrak{S}(\varphi)\right)
\end{equation}
holds.
\end{lem}
\begin{proof}[About the proof]
Over $\CC$ the proof follows from \cite[Lemma 2.1]{Sinclair} and \cite[Lemma 1]{Jewell-Sinclair}. The same proof is valid over any nonarchimedean field $K$ provided it is non-trivially valued (which we always assume). This assumption is used to rescale some operator $T_{n}$ by a nonzero constant so that its operator norm is $\leq 1$ (cf. the proof of \cite[Lemma 2.1]{Sinclair}).  
\end{proof}
Let us now specialize to the case when $\varphi:M\to N$ is a $K$-linear map between two finitely generated $A$-modules and $A$ is a commutative noetherian $K$-algebra over a nonarchimedean field $K$. We fix this setting and the notation for the rest of this section.
\begin{lem}\label{Jewell-Sinclair-Nakayama}
Let $a\in A$ and assume that $[\varphi,a]$ is continuous. Then there exists $b\in A$ and an integer $k\geq0$ such that
\[
(1-ab)a^{k}\mathfrak{S}(\varphi)=0.
\]
\end{lem}
\begin{proof}
With the notation of Lemma \ref{Jewell-Sinclair} let us set $L_{n}$ (resp. $R_{n}$) to be the multiplication by $a$ on $M$ (resp. $N$) for all $n\geq0$. These maps are continuous and, by assumption, so is $L_{n}\varphi-\varphi R_{n}=-[\varphi,a]$. From Lemma \ref{Jewell-Sinclair} we have for some $k\geq0$ the equality
\begin{equation}\label{k=k+1}
\textnormal{cl}(a^{k+1}\mathfrak{S}(\varphi))=\textnormal{cl}(a^{k}\mathfrak{S}(\varphi)).
\end{equation}
Let $N'\subset N$ be the $A$-submodule generated by $\mathfrak{S}(\varphi)$. Then $a^{t}N'$ is closed in $N$ for all $t\geq0$ by Proposition \ref{Fundamental_Finite_Modules} (3) and therefore
\begin{equation}\label{separatingmodule}
a^{t}N'=\left\langle\textnormal{cl}(a^{t}\mathfrak{S}(\varphi))\right\rangle
\end{equation}
(here, we write $\langle S\rangle$ for the submodule of $N$ generated by a subset $S\subset N$). Combining equations (\ref{k=k+1}) and (\ref{separatingmodule}) we obtain that
\[
a^{k+1}N'=a^{k}N'.
\]
The lemma follows immediately from Nakayama's lemma applied to the ideal $I=(a)$ and the finitely generated (as $A$ is noetherian) $A$-module $a^{k}N'$.
\end{proof}
We exploit further the module structure on $M$ and $N$. The \textit{continuity ideal} of $\varphi$ is defined as
\begin{equation}\label{ContinuityIdeal}
\mathfrak{C}(\varphi)=\left\{a\in A:\varphi a\textnormal{ is continuous}\right\}.
\end{equation}
Here $\varphi a$ is the $K$-linear map $m\mapsto \varphi(am)$. It is straightforward that $\mathfrak{C}(\varphi)$ is an ideal in $A$.
\begin{lem}\label{ContinuityIdealLemma}
The restriction of $\varphi$ to the submodule $\mathfrak{C}(\varphi)M$ is continuous.
\end{lem}
\begin{proof}
The noetherianity assumption ensures that $\mathfrak{C}(\varphi)$ is finitely generated. Let us pick generators $a_{1},\dots, a_{r}$ of $\mathfrak{C}(\varphi)$ and $m_{1},\dots,m_{s}$ of $M$. Consider the free $A$-module $F=\bigoplus_{1\leq i\leq r,1\leq j\leq s}A e_{ij}$ equipped with the supremum norm
\[
|\sum_{i,j}b_{ij}e_{ij}|=\max|b_{ij}|.
\]
By Proposition \ref{Fundamental_Finite_Modules} the topology on $\mathfrak{C}(\varphi)M$ agrees with the quotient topology for the projection $\pi:F\to \mathfrak{C}(\varphi)M$ defined on the basis by $\pi(e_{ij})=a_{i}m_{j}$. Thus to prove the lemma we only need to show that the composition
\begin{equation*}
\begin{tikzcd}
F\arrow{r}{\pi}&\mathfrak{C}(\varphi)M\arrow{r}{\varphi}&N
\end{tikzcd}
\end{equation*}
is continuous. By the closed graph theorem we only need to show that if $x_{n}$ is a sequence of elements of $F$ with $\lim_{n\to\infty}x_{n}=0$ then $\lim_{n\to\infty}\varphi\pi(x_{n})=0$. Let us write
\[
x_{n}=\sum_{i,j} b_{ij}^{n}e_{ij}. 
\]
Note that the norm on $F$ enforces that for every fixed pair $(i,j)$ we have
\begin{equation}\label{limzero}
\lim_{n\to\infty}b^{n}_{ij}=0.
\end{equation}
Therefore
\begin{equation*}
\begin{split}
\lim_{n\to\infty}\varphi\pi(x_{n})=&
\lim_{n\to\infty}\varphi\pi\left(\sum _{i,j}b_{ij}^{n}e_{ij}\right)\\
=&
\lim_{n\to\infty}\varphi\left(\sum_{i,j}b_{ij}^{n}a_{i}m_{j}\right)\\
=&\lim_{n\to\infty}\sum_{i,j}\varphi a_{i}\left( b_{ij}^{n}m_{j}\right)=0.
\end{split}
\end{equation*}
To see that the last equality holds it suffices to notice that $\varphi a_{i}$ are continuous by assumption, $\lim_{n\to\infty} b_{ij}^{n}m_{j}=0$ by equality (\ref{limzero}) and the sum $\sum_{i,j}$ is taken over a finite set.
\end{proof}
\begin{rmk}\label{ContinuityIdealMeasuresDiscontinuity}
Because of Lemma \ref{ContinuityIdealLemma} we may think intuitively of $\mathfrak{C}(\varphi)$ as of yet another measure of discontinuity of $\varphi$. If this ideal is `large' then so is the submodule $\mathfrak{C}(\varphi)M\subset M$  and thus the discontinuity of $\varphi$ is `small' in the sense that $\varphi$ becomes continuous after restricting the source to a `large' submodule. 
\end{rmk}
As a corollary we get
\begin{lem}\label{CofiniteContinuityIdeal}
Assume that $\dim_{K}A/\mathfrak{C}(\varphi)<\infty$. Then $\varphi$ is continuous.
\end{lem}
\begin{proof}
Since $M$ is a finitely generated $A$-module, the $A/\mathfrak{C}(\varphi)$-module $M/\mathfrak{C}(\varphi)M$ is also finitely generated. If $A/\mathfrak{C}(\varphi)$ is a finite-dimensional $K$-vector space then it follows that $\dim_{K}M/\mathfrak{C}(\varphi)M<\infty$, i.e., that $\mathfrak{C}(\varphi)M$ is a $K$-vector subspace of $M$ of finite codimension. Since it is also a submodule of $M$, it is closed by Proposition \ref{Fundamental_Finite_Modules}. Therefore, the claim follows from Lemma \ref{FiniteCodimensionContinuity} because by Lemma \ref{ContinuityIdealLemma} the restriction $\varphi:\mathfrak{C}(\varphi)M\to N$ is continuous.
\end{proof}

The lemma below is obvious, and its proof is omitted. 
\begin{lem}\label{Continuity_Ideal_Contains_Ann(M)}
We have $\textnormal{Ann}_{A}(M)\subset\mathfrak{C}(\varphi)$.
\end{lem}

We can also consider the `left version' of $\mathfrak{C}(\varphi)$, which, as we will soon demonstrate, is nothing else than the annihilator of the separating space of $\varphi$. We set
\begin{equation}\label{AnnihilatorSpace}
\mathfrak{A}(\varphi)=\left\{a\in A :a\varphi\textnormal{ is continuous}\right\},
\end{equation}
where $a\varphi$ denotes the $K$-linear map $m\mapsto a\varphi(m)$. This is clearly an ideal in $A$ and it enjoys the following properties.
\begin{lem}\label{A(varphi)}
With the notation above:
\begin{enumerate}

\item $\mathfrak{A}(\varphi)$ equals to the annihilator of the separating space of $\varphi$, i.e.,
\[
\mathfrak{A}(\varphi)=\textnormal{Ann}_{A}(\mathfrak{S}(\varphi))=\left\{a\in A: ay=0\textnormal{ for all }y\in\mathfrak{S}(\varphi)\right\}.
\]
\item If $[\varphi,a]$ is continuous for all $a\in A$ then $\mathfrak{A}(\varphi)=\mathfrak{C}(\varphi)$.
\end{enumerate}
\end{lem}

\begin{proof}
(1) Let $y\in \mathfrak{S}(\varphi)$ and let $\{x_{n}\}\subset M$ be a sequence such that $\lim_{n\to\infty}x_{n}$ and $\lim_{n\to\infty}\varphi(x_{n})=y$. If $a\varphi$ is continuous then
\[
0=\lim_{n\to\infty}a\varphi(x_{n})=ay,
\]
i.e., $\mathfrak{A}(\varphi)\subset \textnormal{Ann}_{A}(\mathfrak{S}(\varphi))$. Conversely, assume that $a\mathfrak{S}(\varphi)=0$. Then $a\varphi$ is continuous by Lemma \ref{SeparatingProperties} (2) and therefore $a\in\mathfrak{A}(\varphi)$.
\medskip

(2) We have $\varphi a=[\varphi,a]+a\varphi$ and thus if $[\varphi,a]$ is continuous then $\varphi a$ is continuous if and only if $a\varphi$ is.
\end{proof}

\begin{rmk}
All the constructions introduced in this section are classically studied in the complex functional analysis, in particular, in the automatic continuity theory. The separating space (\ref{SeparatinSpace}) has been introduced by C.\ E.\ Rickard in \cite{Rickart} to study the problem of the uniqueness of the Banach algebra structure on a given Banach algebra. The stabilization property for the separating space (Lemma \ref{Jewell-Sinclair}) appeared first in a more restrictive variant in the work of A.\ M.\ Sinclair \cite{Sinclair} and in the generality stated above in the work of Jewell--Sinclair \cite{Jewell-Sinclair}, where the authors also acknowledge the contribution of M.\ P.\ Thomas. The continuity ideal (\ref{ContinuityIdeal}) and its left analogue (\ref{AnnihilatorSpace}) have been studied by Bade--Curtis, Jr. \cite{Bade-Curtis}, and the authors attribute the introduction of these objects to J.\ R.\ Ringrose \cite{Ringrose}. Finally, we mention that Lemma \ref{ContinuityIdealLemma} may fail without the noetherianity assumption (cf. \cite[Example 1]{Bade-Curtis}). Similarly, Lemma \ref{Jewell-Sinclair-Nakayama} exploits the noetherianity assumption. These two lemmas does not seem to appear in the classical literature.
\end{rmk}

\section{continuity of derivations for field extensions}
Let $K$ be a nonarchimedean field of characteristic zero (which we assume to be non-trivially valued cf. Assumption \ref{Nontrivial_Assumption}) and let $L$ be a Banach $K$-algebra which is a field. In this section we want to investigate under what assumptions on the field extension $K\subset L$ it is true that every $K$-linear derivation of $L$ is continuous. We recall (cf. Remark (\ref{Kedlaya-Remark})) that $L$ does not need to be nonarchimedean, i.e., that we cannot assume that the norm on $L$ is multiplicative. This causes various technical difficulties, but we are able to prove the following proposition.
\begin{prop}\label{DD_p=0}
In the notation and under the assumptions above the following conditions are equivalent:
\begin{enumerate}
    \item Every $K$-linear derivation of $L$ is continuous,

    \item every $K$-linear derivation of $L$ is zero,

    \item the extension $K\subset L$ is algebraic,

    \item the extension $K\subset L$ is finite.
\end{enumerate}
Moreover, if one of the above conditions holds, then (up to equivalence of $K$-algebra norms on $L$) $L$ is a nonarchimedean field.
\end{prop}
\noindent
The proof will be concluded from a series of lemmas. The main technical result of this section (Lemma \ref{NormsOnFiniteExtensions} below) has nothing to do with derivations. It is very well known (see \cite[Lemma 1, Section 3.4.3]{BGR}) when $L$ is nonarchimedean. We show that the nonarchimedeanity assumption is superfluous.
\begin{lem}\label{NormsOnFiniteExtensions}
Assume that the extension $K\subset L$ is algebraic. Then $[L:K]<\infty$ and $L$ is nonarchimedean.
\end{lem}
\begin{proof}
Let us fix a nonarchimedean norm $|\ |$ on $K$ and let $\| \ \|$ be a Banach $K$-algebra norm on $L$ compatible with this norm. It is well-known (cf. \cite[Theorem 2, Section 3.2.4]{BGR}) that for every algebraic extension $K\subset K'$ there exists a unique multiplicative $K$-algebra norm on $K'$. In particular, if we fix an algebraic closure $\overline{K}$ of $K$ containing $L$, then there exists a unique multiplicative $K$-norm on $\overline{K}$ that we denote $| \ |$. For $a\in L$ the \textit{spectral radius} of $a$, defined as
\begin{equation}\label{Spectral_Radius}
\rho(a)\stackrel{\textnormal{def.}}{=}\lim_{n\to\infty}\|a^{n}\|^{\frac{1}{n}}=\inf_{n\geq1}\|a^{n}\|^{\frac{1}{n}}
\end{equation}
depends only on the equivalence class of the norm. The second equality above follows from Fekete's lemma \cite[Lemma 6.1.4]{Kedlaya_book}. Since by Proposition \ref{FiniteNorms} any two nonarchimedean $K$-norms on the finite dimensional $K$-vector space $K[a]$ are equivalent, we see that
\begin{equation}\label{NormInequality}
|a|=\rho(a)\leq\|a\|.
\end{equation}
In particular,
\begin{equation}\label{ContinuousEmbedding}
\textit{The inclusion }(L,\| \ \|)\to (\overline{K}, | \ |) \textit{ is continuous}.
\end{equation}
We now mimic the proof of \cite[Lemma 1, Section 3.4.3]{BGR} to show that $[L:K]<\infty$ . For an element $a\in \overline{K}\setminus K$ we let $f\in K[x]$ be the minimal polynomial of $a$ and we define as in \cite[Section 3.4.2, page 148]{BGR}
\begin{equation}\label{Root_distance}
r(a)\stackrel{\textnormal{def.}}{=}\min_{f(b)=0,b\neq a}|a-b|.
\end{equation}
If $[L:K]=\infty$, then there exists an infinite sequence  $1=x_{0},x_{1},x_{2},\dots$ of elements of $L$ linearly independent over $K$, and nonzero scalars $c_{0},c_{1},c_{2},\dots\in K$, such that
\begin{enumerate}
    \item $| c_{i+1}x_{i+1}|\leq | c_{i}x_{i} |$, $\lim _{i\to\infty}\|c_{i}x_{i}\|=0$,

    \item $ |c_{i+1}x_{i+1} |\leq r\left(\sum_{j=1}^{i} c_{j}x_{j}\right)$.
\end{enumerate}
Then (\ref{NormInequality}) implies that $\lim_{i\to\infty}|c_{i}x_{i}|=0$ and therefore, by (the proof of) \cite[Lemma 1, Section 3.4.3]{BGR} the sequence
\[
a_{k}=\sum_{i=0}^{k}c_{i}x_{i}
\]
does not have a limit in $(\overline{K},| \ |)$. Therefore, $(L,\| \ \|)$ cannot be complete, because if it were, then the assumption $\lim _{i\to\infty}\|c_{i}x_{i}\|=0$ would force $a_{k}$ to have a limit in $L$ and this would contradict (\ref{ContinuousEmbedding}). Finally, if $L$ is a finite extension of $K$ then the unique multiplicative $K$-algebra norm on $L$ is equivalent to $\| \ \|$ by Proposition \ref{FiniteNorms} and therefore $L$ is nonarchimedean.
\end{proof}
\begin{rmk}
We do not know if the above lemma holds when $\textnormal{char }K=p>0$. Note that if the extension $K\subset L$ is purely inseparable then the value (\ref{Root_distance}) is not well-defined and hence the above proof does not carry to the positive characteristic. In \cite[Lemma 1, Section 3.4.3]{BGR} it is shown in any characteristic that if the extension $K\subset\overline{K}$ is infinite then $\overline{K}$ is not nonarchimedean with respect to the unique multiplicative nonarchimedean norm extending the one on $K$. The proof is reduced to separable extensions by observing that if $K\subset \overline{K}$ is infinite the so is $K\subset K_{\textnormal{sep}}$, where $K_{\textnormal{sep}}$ is the separable closure of $K$.
\end{rmk}
The following is well known and remains true for any extension of fields in characteristic zero. We omit the proof.
\begin{lem}\label{Extending_Derivations}
Let $K\subset K'\subset L$ be fields of characteristic zero. 
\begin{enumerate}
\item Any $K$-linear derivation of $K'$ can be extended to a $K$-linear derivation of $L$.
\item If $x\in L$ is transcendental over $K$, then there exists a $K$-linear dervation $\delta_{x}:L\to L$ such that $\delta_{x}(x)=1$.
\end{enumerate}
\end{lem}
We can now prove Proposition \ref{DD_p=0}.

\begin{proof}[Proof of Proposition \ref{DD_p=0}]
Clearly (4)$\implies$(3) and since $\textnormal{char }K=0$ (3)$\implies$(2). The implication (2)$\implies$(1) is also clear and by Lemma \ref{NormsOnFiniteExtensions} (3)$\implies$(4). The moreover part follows from Lemma \ref{NormsOnFiniteExtensions}. Therefore, we only need to show that (1)$\implies$(3).
\medskip

Suppose that the extension $K\subset L$ is not algebraic. We will demonstrate that there exists a discontinuous $K$-linear derivation of $L$. Let $K_{1}$ denote the integral closure of $K$ in $L$. It is a normed $K$-algebra under the norm restricted from $L$ and it is a field. We consider two cases.
\medskip

\noindent
\textit{Case 1: $K_{1}$ is not complete}. Let $x\in L$ be an element that lies in the topological closure of $K_{1}$ but not in $K_{1}$ itself. By assumption this element is transcendental over $K_{1}$ and thus by Lemma \ref{Extending_Derivations} there exists $K_{1}$-linear (and in particular $K$-linear) derivation $\delta_{x}:L\to L$ such that $\delta_{x}(x)=1$. Such derivation cannot be continuous because $\delta_{x|K_{1}}=0$ and $x$ is in the closure of $K_{1}$.
\medskip

\noindent
\textit{Case 2: $K_{1}$ is complete}. In this situation by Lemma \ref{NormsOnFiniteExtensions} the extension $K\subset K_{1}$ has to be finite and moreover $K_{1}$ is nonarchimedean. Since the extension $K_{1}\subset L$ is still infinite we may replace $K$ with $K_{1}$ and we are reduced to the case when $K=K_{1}$ is integrally closed in $L$. Let $x\in L\setminus K$ be any element. Since the norm on $K$ is non-trivial, we may replace $x$ by $\pi x$ for some $\pi \in K^{\times}$ with $|\pi|$ small enough to ensure that the sequence
\[
a_{k}=\sum_{n=0}^{k}\frac{x^{n}}{n!}
\]
is a Cauchy sequence in $K(x)$. We set
\[
\exp(x)=\lim_{k\to\infty} a_{k}\in L.
\]
Since $x$ is transcendental over $K$, the field $K(x)$ is abstractly isomorphic to the field of rational functions in one variable over $K$. Let $\partial$ denote the $K$-linear derivation of $K(x)$ defined by $\partial(x)=1$. By lemma \ref{Extending_Derivations} we may extend $\partial$ to a derivation of $K(x)(\exp(x))$ and further to a derivation of $L$. If $\partial$ is a discontinuous derivation of $K(x)(\exp(x))$, then we reach a contradiction. Therefore, we may assume that it is continuous and therefore
\[
\partial(a_{k})=a_{k-1},\qquad\partial(\exp(x))=\exp(x).
\]
However, it is well-known (and follows from a simple computation with degrees of rational functions) that if $(K(x),\partial)\subset D$ is an extension of differential fields then any element $f\in D$ which satisfies the differential equation
\[
\partial f=cf,\qquad c\in\ZZ\setminus\{0\}
\]
is transcendental over $K(x)$. It follows that $\exp(x)$ is transcendental over $K(x)$. On the other hand, it lies in the topological closure of $K(x)$ in $L$. By Lemma \ref{Extending_Derivations} there exists a $K(x)$-linear derivation $\delta:L\to L$ with $\delta(\exp(x))=1$. This derivation cannot be continuous, so we are done. 
\end{proof}

\section{Discontinuous differential operators}

In this section, we use Proposition \ref{DD_p=0} to `construct' discontinuous differential operators. For the rest of this section we fix a nonarchimedean field $K$ of characteristic zero, a commutative, noetherian Banach $K$-algebra $A$ and two finitely generated $A$-modules $M,\ N$.
\begin{prop}\label{DDD=0}
With the notation and under the assumptions above assume moreover that there exists a maximal ideal $\mm\subset A$ with the following properties:
\begin{enumerate}
    \item $\mm\in \textnormal{Supp}(M)\cap\textnormal{Ass}_{A}(N),$

    \item $[A/\mm:K]=\infty$.
\end{enumerate}
Then there exists a discontinuous $K$-linear differential operator $P:M\to N$.
\end{prop}
\begin{rmk}\label{1stImplication}
Clearly, the above proposition proves the implication (2)$\implies$(1) in Theorem \ref{MainThm2}. Moreover, given any maximal ideal $\mm\subset A$ we may take $M=A,\ N=A/\mm$ to force the condition $\mm\in\textnormal{Supp}(M)\cap\textnormal{Ass}_{A}(M)$ to hold. Therefore, the proposition also proves the implication (1)$\implies$(2) in Theorem \ref{MainThm1}.
\end{rmk}
\begin{proof}[Proof of Proposition \ref{DDD=0}]
Let $\mm\subset A$ be a maximal ideal with $\mm\in\textnormal{Supp}(M)\cap\textnormal{Ass}_{A}(N)$. By Lemma \ref{SupCupAss} we have an $A$-linear injection
\[
\epsilon:A/\mm\to N,
\]
and an $A$-linear surjection
\[
\pi:M\to A/\mm.
\]
If moreover $[A/\mm:K]=\infty$ then it follows from Proposition \ref{DD_p=0} that there exists a discontinuous $K$-linear derivation 
\[
\delta:A/\mm\to A/\mm.
\]
By Lemma \ref{CompositionLemma} 
\[
P=\epsilon\delta\pi
\]
is a $K$-linear differential operator and we are  working to show that it is \textit{not} continuous. As $\epsilon$ is $A$-linear, it is a homeomorphism onto a closed subset by Proposition \ref{Fundamental_Finite_Modules} (3) and thus $P$ is continuous if and only if $\delta\pi$ is continuous. On the other hand, by Proposition \ref{Fundamental_Finite_Modules} (4), $\pi$ is a quotient and therefore $\delta\pi$ is continuous if and only if $\delta$ is continuous. This shows that $P$ is not continuous.
\end{proof}
\section{A criterion for automatic continuity}

In this section, we prove a criterion for automatic continuity of differential operators which provides the (1)$\implies$(2) part of Theorem \ref{MainThm2} and therefore also the (2)$\implies$(1) part of Theorem \ref{MainThm1}. For the rest of the section we fix a nonarchimedean field $K$ (of arbitrary characteristic), a commutative, noetherian, Banach $K$-algebra $A$, and two finitely generated $A$-modules $M,\ N$.

\begin{prop}\label{ACC1}
With the above notation assume moreover that there exists a discontinuous differential operator from $M$ to $N$. Then there exists a maximal ideal $\mm\subset A$ with the following properties:
\begin{enumerate}

\item $\mm\in\textnormal{Supp}(M)\cap\textnormal{Ass}_{A}(N)$,

\item $[A/\mm:K]=\infty$.
\end{enumerate}
\end{prop}
We now prove the above proposition. We assume that there exists a discontinuous differential operator in $\cD_{A/K}(M,N)$. Then there exists a minimal integer $d$ such that $\cD_{A/K}^{\leq d}(M,N)$ contains a discontinuous differential operator. This integer $d$ remains fixed for the rest of this section. Recall that for a $K$-linear map $\varphi:M\to N$ we can consider its separating space $\mathfrak{S}(\varphi)$, its continuity ideal $\mathfrak{C}(\varphi)$, and the annihilator $\mathfrak{A}(\varphi)$ introduced in the previous section as (\ref{SeparatinSpace}), (\ref{ContinuityIdeal}), and (\ref{AnnihilatorSpace}) respectively. These objects are crucial in the proof of Proposition \ref{ACC1}. Consider the family of ideals of $A$
\[
\textnormal{Discont}_{d}(M,N)=\{\mathfrak{C}(\varphi):\varphi\in\cD^{\leq d}_{A/K}(M,N)\textnormal{ is discontinuous}\}.
\]
\begin{lem}\label{PrimeDiscontinuousIdeal}
With the notation and under the assumptions above:
\begin{enumerate}
\item $\textnormal{Discont}_{d}(M,N)$ contains a maximal element.
\item Every maximal element of $\textnormal{Discont}_{d}(M,N)$ is a prime ideal of $A$.
\end{enumerate}
\end{lem}

\begin{proof}
By assumption $\textnormal{Discont}_{d}(M,N)$ is nonempty. Therefore (1) follows from the noetherianity of $A$ as every nonempty family of ideals in a noetherian ring has a maximal element. To prove (2) let us fix $\varphi_{0}$ such that $\mathfrak{C}(\varphi_{0})$ is maximal in $\textnormal{Discont}_{d}(M,N)$ and assume that $ab\in\mathfrak{C}(\varphi_{0})$ but $a\notin\mathfrak{C}(\varphi_{0})$. This means that $\varphi_{0}a$ is a discontinuous element of $\cD_{A/K}^{\leq d}(M,N)$ and thus $\mathfrak{C}(\varphi_{0}a)\in\textnormal{Discont}_{d}(M,N)$. Clearly $\mathfrak{C}(\varphi_{0})\subset \mathfrak{C}(\varphi_{0}a)$ because if for some $c\in A$ the map $\varphi_{0}c$ is continuous then so is $\varphi_{0}ca$. By maximality we get that $\mathfrak{C}(\varphi_{0})= \mathfrak{C}(\varphi_{0}a)$. Since $\varphi_{0}ab$ is continuous by the initial assumption, we conclude that $b\in\mathfrak{C}(\varphi_{0}a)=\mathfrak{C}(\varphi_{0})$ and thus $\mathfrak{C}(\varphi_{0})$ is prime.
\end{proof}

\begin{rmk}[cf. Remark \ref{ContinuityIdealMeasuresDiscontinuity}]
Intuitively $\varphi_{0}$ introduced in the proof of the above Lemma is the `least discontinuous' element among discontinuous elements of $\cD_{A/K}^{\leq d}(M,N)$.
\end{rmk}

\begin{proof}[Proof of Proposition \ref{ACC1}]
Let $\varphi\in\cD_{A/K}^{\leq d}(M,N)$ be a discontinuous operator such that $\mathfrak{C}(\varphi)$ is a maximal element of $\textnormal{Discont}_{d}(M,N)$. Denote $\mm=\mathfrak{C}(\varphi)$. We are working to show that $\mm$ is a maximal ideal of $A$ with $\dim_{K}A/\mm=\infty$, which is contained in $\textnormal{Supp}(M)\cap\textnormal{Ass}_{A}(N)$. By lemma \ref{PrimeDiscontinuousIdeal} we already know that $\mm$ is a prime ideal. Note that our choice of $d$ ensures that every element of $\cD_{A/K}^{\leq d-1}(M,N)$ is continuous and, in particular, $[\varphi,a]$ is continuous for every $a\in A$. From Lemma \ref{A(varphi)} we conclude that 
\begin{equation}\label{m=Ann}
\mm=\mathfrak{C}(\varphi)=\mathfrak{A}(\varphi)=\textnormal{Ann}_{A}(\mathfrak{S}(\varphi)).
\end{equation}
The proof breaks into several steps.
\medskip

\noindent
\textit{Step 1. $\mm$ is maximal}. Assume conversely. Since $\mm$ is not maximal, there exists a nonzero element $\overline{a}\in A/\mm$ which is not a unit. Let $a\in A$ be a lift of this element.  We have already noticed that $[\varphi,a]$ is continuous and therefore by Lemma \ref{Jewell-Sinclair-Nakayama} there exists $b\in A$ and an integer $k\geq0$ such that 
\[
(1-ab)a^{k}\mathfrak{S}(\varphi)=0.
\]
By (\ref{m=Ann}) we have $(1-ab)a^{k}\in\mm$ and since this ideal is prime, either $a\in \mm$ or $(1-ab)\in\mm$. However, neither of these two possibilities may occur because then $\overline{a}$ would be either zero or a unit in $A/\mm$ contradicting the initial choice of $\overline{a}$. This shows that $\mm$ must be a maximal ideal of $A$.
\medskip

\noindent
\textit{Step 2. $\dim_{k}A/\mm=\infty$}. Since $\mm$ is the continuity ideal of $\varphi$, if this dimension was finite then $\varphi$ would be continuous by Lemma \ref{CofiniteContinuityIdeal}.
\medskip

\noindent
\textit{Step 3. $\mm\in\textnormal{Supp}(M)\cap \textnormal{Ass}_{A}(N)$}. To show that $\mm$ is contained in the support of $M$ it suffices to show that $\textnormal{Ann}_{A}(M)\subset \mm$ (cf. Section \ref{AlgebraicPleriminariesSection}). This follows from Lemma \ref{Continuity_Ideal_Contains_Ann(M)} because $\mm$ is the continuity ideal of $\varphi$. To show that $\mm$ is an associated prime of $N$ it suffices to find an element $y\in N$ with $\textnormal{Ann}_{A}(y)=\mm$. To do so let us take $y$ to be any nonzero element in $\mathfrak{S}(\varphi)$. Such element must exist because otherwise $\varphi$ would be continuous by Lemma~\ref{SeparatingProperties}~(1). Since $\mm=\textnormal{Ann}_{A}(\mathfrak{S}(\varphi))$ by (\ref{m=Ann}), we see that $\mm y=0$ and therefore $\mm\subset \textnormal{Ann}_{A}(y)$. However, $\mm$ is maximal and the annihilator of $y$ is a proper ideal of $A$ because the unit $1\in A$ does not kill $y$. Therefore, $\mm=\textnormal{Ann}_{A}(y)$ and we are done.  
\end{proof}

\section{Examples}

In the final section, we provide some examples of where the results of this note may be applied. Throughout this section $K$ is a nonarchimedean field of arbitrary characteristic.
\medskip

\begin{ex}[Tate vs. Berkovich]\label{Tate_vs_Berkovich}
If $A$ is an affinoid $K$-algebra in the sense of Tate (cf. \cite[Definition 1, Section 6.1.1]{BGR}) then it is known (see \cite[Section 6.1.2, Corollary 3]{BGR}) that $A/\mm$ is a finite extension of $K$ for every maximal $\mm\subset A$. Therefore, if $A$ is affinoid in the sense of Tate then every $K$-linear differential operator between any two finitely generated $A$-modules is continuous by Theorem \ref{MainThm1}. On the other hand, If $A$ is affinoid in the broader sense of Berkovich (cf. \cite[Definition 2.1.1]{Berkovich_book}) then it may happen that $A/\mm$ is not a finite extension of $K$ (see \cite[Remark 2.12]{Kedlaya_fields}) and therefore continuity of differential operators is not automatic in this setting.
\end{ex}

\begin{ex}[Domains]
If a $K$-algebra $A$ is a domain which is not a field, $M,N$ are finitely generated $A$-modules and $N$ is torsion-free, then any $K$-linear differential operator $P:M\to N$ is continuous. This claim follows from Theorem \ref{MainThm2} because then $(0)$ is the only associated prime of $N$ and it is not a maximal ideal of $A$. Note that the characteristic of $K$ is arbitrary.
\end{ex}

\begin{ex}[discontinuous derivations]
The following provides an example of a phenomenon that at the first may seem counter-intuitive. By Lemma \ref{DD_p=0} there exists a discontinuous $\QQ_{p}$-linear derivation of $\CC_{p}$, and it is an exercise to show that in fact every $\QQ_{p}$-linear derivation of $\CC_{p}$ is either zero or discontinuous. On the other hand, every $\QQ_{p}$-linear derivation of the Tate algebra $\CC_{p}\langle x\rangle$ is continuous by the previous example.
\end{ex}

\begin{ex}[Complex setting]
Theorem \ref{MainThm1} holds also for complex Banach algebras. Indeed, a careful reader may have observed that in the proof of of implication (2)$\implies$(1) we have merely used the noetherianity of the Banach $K$-algebra $A$ (and the nonarchimedeanity of $K$ does not play a role), and the implication (1)$\implies$(2) is automatic over $\CC$ by the Gelfand--Mazur theorem. However, in this case, Theorem \ref{MainThm1} is not very interesting because by \cite[Satz 2, Page 55]{GR_book} every noetherian Banach algebra over $\CC$ has finite dimension as a $\CC$-vector space, and of course every linear transformation between $\CC$-vector spaces of finite dimension is continuous.
\end{ex}

\begin{ex}[Positive characteristic]
It would be interesting to find an analogue of Theorem \ref{MainThm1} in positive characteristic. Here, we show that some modification is needed because as it is stated, Theorem \ref{MainThm1} fails if $\textnormal{char }K=p>0$. Indeed, take $K=\FF_{p}(\!(x)\!)$ to be the field of formal Laurent series with $\FF_{p}$ coefficients and let $L$ be the completed algebraic closure of $K$. Then by \cite[Section 3.4.1, Proposition 3]{BGR} $L$ is algebraically closed, and thus any $K$-linear derivation of $L$ is zero because $L^{p}=L$. On the other hand, $K\subset L$ is not finite (nor even algebraic). Therefore the implication (1)$\implies$(2) of Theorem \ref{MainThm1} cannot hold in positive characteristic.
\end{ex}

\bibliographystyle{plain}
\bibliography{Bibliography}

\begin{thebibliography}{10}

\bibitem{Ardakov3}
Konstantin Ardakov, Andreas Bode, and Simon Wadsley.
\newblock {$\wideparen{D}$}-modules on rigid analytic spaces {III}: weak holonomicity and operations.
\newblock {\em Compos. Math.}, 157(12):2553--2584, 2021.

\bibitem{Ardakov2}
Konstantin Ardakov and Simon Wadsley.
\newblock {$\wideparen{D}$}-modules on rigid analytic spaces {II}: {K}ashiwara's equivalence.
\newblock {\em J. Algebraic Geom.}, 27(4):647--701, 2018.

\bibitem{Ardakov1}
Konstantin Ardakov and Simon~J. Wadsley.
\newblock {$\wideparen{D}$}-modules on rigid analytic spaces {I}.
\newblock {\em J. Reine Angew. Math.}, 747:221--275, 2019.

\bibitem{Bade-Curtis}
W.~G. Bade and P.~C. Curtis, Jr.
\newblock The continuity of derivations of {B}anach algebras.
\newblock {\em J. Functional Analysis}, 16:372--387, 1974.

\bibitem{Berkovich_book}
Vladimir~G. Berkovich.
\newblock {\em Spectral theory and analytic geometry over non-{A}rchimedean fields}, volume~33 of {\em Mathematical Surveys and Monographs}.
\newblock American Mathematical Society, Providence, RI, 1990.

\bibitem{Berthelot}
Pierre Berthelot.
\newblock {${\mathscr D}$}-modules arithm\'etiques. {I}. {O}p\'erateurs diff\'erentiels de niveau fini.
\newblock {\em Ann. Sci. \'Ecole Norm. Sup. (4)}, 29(2):185--272, 1996.

\bibitem{BGR}
S.~Bosch, U.~G\"{u}ntzer, and R.~Remmert.
\newblock {\em Non-{A}rchimedean analysis}, volume 261 of {\em Grundlehren der mathematischen Wissenschaften [Fundamental Principles of Mathematical Sciences]}.
\newblock Springer-Verlag, Berlin, 1984.
\newblock A systematic approach to rigid analytic geometry.

\bibitem{Camargo}
Juan Esteban~Rodríguez Camargo.
\newblock The analytic de {R}ham stack in rigid geometry, arxiv:2401.07738, 2024.

\bibitem{Curtis}
Philip~C. Curtis, Jr.
\newblock Derivations of commutative {B}anach algebras.
\newblock {\em Bull. Amer. Math. Soc.}, 67:271--273, 1961.

\bibitem{GR_book}
H.~Grauert and R.~Remmert.
\newblock {\em Analytische {S}tellenalgebren}, volume Band 176 of {\em Die Grundlehren der mathematischen Wissenschaften}.
\newblock Springer-Verlag, Berlin-New York, 1971.
\newblock Unter Mitarbeit von O. Riemenschneider.

\bibitem{Jewell-Sinclair}
Nicholas~P. Jewell and Allan~M. Sinclair.
\newblock Epimorphisms and derivations on {$L^1(0,1)$} are continuous.
\newblock {\em Bull. London Math. Soc.}, 8(2):135--139, 1976.

\bibitem{Johnson}
B.~E. Johnson.
\newblock Continuity of derivations on commutative algebras.
\newblock {\em Amer. J. Math.}, 91:1--10, 1969.

\bibitem{Kedlaya_book}
Kiran~S. Kedlaya.
\newblock {\em {$p$}-adic differential equations}, volume 125 of {\em Cambridge Studies in Advanced Mathematics}.
\newblock Cambridge University Press, Cambridge, 2010.

\bibitem{Kedlaya_fields}
Kiran~S. Kedlaya.
\newblock On commutative nonarchimedean {B}anach fields.
\newblock {\em Doc. Math.}, 23:171--188, 2018.

\bibitem{Meb2}
Z.~Mebkhout and L.~Narv\'{a}ez-Macarro.
\newblock La th\'{e}orie du polyn\^{o}me de {B}ernstein-{S}ato pour les alg\`ebres de {T}ate et de {D}work-{M}onsky-{W}ashnitzer.
\newblock {\em Ann. Sci. \'{E}cole Norm. Sup. (4)}, 24(2):227--256, 1991.

\bibitem{Mihara}
Tomoki Mihara.
\newblock Characterisation of the {B}erkovich spectrum of the {B}anach algebra of bounded continuous functions.
\newblock {\em Doc. Math.}, 19:769--799, 2014.

\bibitem{Rickart}
C.~E. Rickart.
\newblock The uniqueness of norm problem in {B}anach algebras.
\newblock {\em Ann. of Math. (2)}, 51:615--628, 1950.

\bibitem{Ringrose}
J.~R. Ringrose.
\newblock Automatic continuity of derivations of operator algebras.
\newblock {\em J. London Math. Soc. (2)}, 5:432--438, 1972.

\bibitem{Schikhof}
W.~H. Schikhof.
\newblock Uniqueness of {B}anach algebra topology for a class of non-{A}rchimedean algebras.
\newblock {\em Nederl. Akad. Wetensch. Indag. Math.}, 46(1):47--49, 1984.

\bibitem{Schneider_book}
Peter Schneider.
\newblock {\em Nonarchimedean functional analysis}.
\newblock Springer Monographs in Mathematics. Springer-Verlag, Berlin, 2002.

\bibitem{Sinclair}
Allan~M. Sinclair.
\newblock Homomorphisms from {$C\sb{0}(R)$}.
\newblock {\em J. London Math. Soc. (2)}, 11(2):165--174, 1975.

\bibitem{SinclairBook}
Allan~M. Sinclair.
\newblock {\em Automatic continuity of linear operators}, volume No. 21 of {\em London Mathematical Society Lecture Note Series}.
\newblock Cambridge University Press, Cambridge-New York-Melbourne, 1976.

\bibitem{Singer-Wermer}
I.~M. Singer and J.~Wermer.
\newblock Derivations on commutative normed algebras.
\newblock {\em Math. Ann.}, 129:260--264, 1955.

\bibitem{stacks-project}
The {Stacks project authors}.
\newblock The stacks project.
\newblock \url{https://stacks.math.columbia.edu}, 2024.

\bibitem{MP_Thomas}
Marc~P. Thomas.
\newblock The image of a derivation is contained in the radical.
\newblock {\em Ann. of Math. (2)}, 128(3):435--460, 1988.

\bibitem{vdP}
M.~van~der Put.
\newblock Non-archimedean function algebras.
\newblock {\em Indag. Math.}, 33:60--77, 1971.
\newblock Nederl. Akad. Wetensch. Proc. Ser. A {\bf 74}.

\bibitem{VdP_derivations}
Marius van~der Put.
\newblock Continuous derivations of valued fields.
\newblock {\em Bull. Soc. Math. France}, 101:71--112, 1973.

\end{thebibliography}

\end{document}